\documentclass[twoside,12pt]{article}
\usepackage{times,amsmath,amssymb,amsthm,cite}
\newcommand{\subj}[1]{\par\noindent{\bf AMS Subject Classifications: }#1.}
\newcommand{\keyw}[1]{\par\noindent{\bf Keywords: }#1.}
\numberwithin{equation}{section}
\numberwithin{figure}{section}
\newtheorem{theorem}{Theorem}[section]
\newtheorem{lemma}[theorem]{Lemma}

\newtheorem{proposition}[theorem]{Proposition}
\theoremstyle{definition}
\newtheorem{definition}[theorem]{Definition}
\newtheorem{example}[theorem]{Example}
\theoremstyle{remark}
\newtheorem{remark}[theorem]{Remark}

\everymath{\displaystyle}
\date{}
\newcommand{\R}{\mathbb{R}}
\newcommand{\T}{\mathbb T}

\newcommand{\Z}{\mathbb Z}

\def\qed{\hbox{${\vcenter{\vbox{
  \hrule height 0.4pt\hbox{\vrule width 0.4pt height 6pt
  \kern5pt\vrule width 0.4pt}\hrule height 0.4pt}}}$}}

\long\def\symbolfootnote[#1]#2{\begingroup%
\def\thefootnote{\fnsymbol{footnote}}\footnote[#1]{#2}\endgroup}
\begin{document}
\title{ \center\Large\bf  Time scales: from Nabla calculus to Delta calculus and vice versa via duality }
\author{{\bf M. Cristina Caputo}\\
The University of Texas at Austin \\
Department of Mathematics\\
\\
\tt }

\maketitle

\begin{abstract}
In this note we show how one can obtain results from the {\it nabla} calculus from
results on the {\it delta} calculus and vice versa via a duality argument.~We provide applications of the main results to the calculus of variations on time scales.

\end{abstract}
\subj{39A10, 26E70, 49K05}
\keyw{Time scales, nabla calculus, delta calculus, calculus of variations}

\bibliographystyle{plain}

\label{lastpage}

\section{Introduction}\label{sec1}
The {\it time scale delta calculus} was introduced for the first time in 1988 by Hilger~\cite{Hil} to unify the theory of difference equations and the theory of differential equations.  It was extensively studied by Bohner~\cite{Boh} and Hilscher and Zeidan~\cite{Hilzei} who introduced the calculus of variations on the {\it time scale delta}  calculus (or simply {\it delta} calculus).~In 2001 the  {\it time scale nabla calculus} (or simply {\it nabla} calculus) was introduced by Atici and Guseinov~\cite{Atici}. 

Both theories of the {\it delta} and the {\it nabla} calculus  can be applied to any field that requires the study of both continuous and discrete data. For instance, the {\it nabla} calculus has been applied to maximization (minimization) problems in economics~\cite{ABL,Atici}. Recently several authors have contributed to the development of the calculus of variations on time scales (for instance, see~\cite{AMC,tormal,tormar}).

To the best of the author's knowledge there is no known technique to obtain results from the {\it nabla} calculus directly from results on the {\it delta} calculus and vice versa. In this note we underline that, in fact, this is possible.
We show that the two types of calculus,  the {\it nabla} and the {\it delta} on time scales, are the ``dual'' of each other. One can reciprocally obtain results
for one type of calculus from the other and vice versa without making any assumptions on the regularity of the time scales (as it was done in~\cite{GuSi}). We prove that results for the {\it nabla} (respectively  the {\it delta}) calculus can be obtained by the dual analogous ones which will be in the  {\it delta} (respectively  {\it nabla}) context. Therefore, if they have already been proven for the  {\it delta} case (respectively  the {\it delta}), it is not necessary to reprove them for the  {\it nabla} setting  (respectively  {\it nabla}).  

This article is organized as follows: in second section we review some basic definitions. In third section we introduce the {\it dual} time scales. In the fourth section we derive a few properties related to duality. In the fifth section we state the Duality Principle, which is the main result of the article, and we apply it to a few examples. Finally, in the last section, we apply the Duality Principle to the calculus of variations on time scales.
\section{ Review of basic definitions}\label{sec2}
We first review some basic definitions and hence introduce both types of calculus (for a complete list of definitions for the~{\it delta} calculus see the pioneering book by Bohner and Peterson~\cite{BohPet}).

A {\it time scale} $\T$ is any closed nonempty subset $\T$  of $\R$.

The {\it jump operators} $ \sigma$,  $\rho:\T\rightarrow\T$ are defined by
$$\sigma(t)=\inf\{s\in\T: s>t\}, \,\,\,\,\, \hbox{and}\,\,\,\rho(t)=\sup\{s\in\T: s<t\}, $$
with $\inf \emptyset :=\sup \T,$ $\sup \emptyset:= \inf \T.$ A point $t\in\T$ is called {\it right-dense} if $\sigma(t)=t$,
{\it right-scattered} if $\sigma(t)>t$,
{\it left-dense} if $\rho(t)=t$, {\it left-scattered} if $\rho(t)<t$.

The {\it forward graininess} $\mu:\T\rightarrow \R$ is defined by $\mu(t)=\sigma(t)-t$, and the {\it backward graininess} $\nu:\T\rightarrow\R$ is defined by $\nu(t)=t-\rho(t)$.

Given a time scale $\T$, we denote $\T^{\kappa}:=\T\setminus (\rho(\sup \T),\sup\T]$, if $\sup\T<\infty$ and  $\T^{\kappa}:=\T$ if  $\sup\T=\infty$. Also $\T_{\kappa}:=\T\setminus [\inf\T, \sigma(\inf\T))$ if $\inf\T>-\infty$ and $\T_{\kappa}=:\T$ if $\inf\T=-\infty$.
In particular, if $a,b\in\T$ with $a<b$, we denote by $[a,b]$ the interval $[a,b]\cap\T$. It follows that
$$[a,b]^{\kappa}=[a,\rho(b)],\,\,\,\,\hbox{and}\,\,\,\, [a,b]_{k}=[\sigma(a),b].$$

Of course, $\R$ itself is one trivial example of time scale, but one could also take $\T$ to be the Cantor set. For more interesting examples of time scales we suggest reading~\cite{BohPet}.

Let $f$ be a function defined on $\T$, we say that:
\begin{definition}
   $f$  is rd-continuous (or right-dense continuous) (we write $f\in C_{rd}$) if it is continuous at the right-dense points and its left-sided limits exist (finite) at all left-dense points; $f$ is ld-continuous (or left-dense continuous) if it is continuous at the left-dense points and its right-sided limits exist (finite) at all right-dense point.\end{definition}
  
\subsection{Definition of derivatives}
\begin{definition}
   A function $f:\T\rightarrow\R$ is said to be {\it delta} differentiable at $t\in \T^{\kappa}$ if
for all $\epsilon>0$ there exists $U$ a neighborhood of $t$ such that for some $\alpha$, the inequality
$$|f(\sigma(t))-f(s)-\alpha(\sigma(t)-s)|  < \epsilon |\sigma(t)-s|,$$
is true for all $s\in U$. We write $f^{\Delta}(t)=\alpha$.
\end{definition}
  
\begin{definition}
   $f:\T\rightarrow\R$ is said to be {\it delta} differentiable on $\T$ if $f:\T\rightarrow\R$ is {\it delta} differentiable for all $t\in \T^{\kappa}$.
\end{definition}
  
It is easy to show that, if $f$ is~{\it delta} differentiable on $\T$, then the following formula holds
$$f^{\sigma}=f+\mu f^{\Delta},$$
where $f^{\sigma}=f \circ \sigma $ (the proof can be found in~\cite{BohPet}).

\begin{definition}
    A function $f:\T\rightarrow\R$ is said to be {\it nabla} differentiable at $t\in \T_{\kappa}$ if for all $\epsilon>0$ there exists $U$ a neighborhood of $t$ such that for some $\beta$, the inequality
$$|f(\rho(t))-f(s)-\beta(\rho(t)-s)|  < \epsilon |\rho(t)-s|,$$
is true for all $s\in U$. We write $f^{\nabla}(t)=\beta$.
\end{definition}
  
\begin{definition}
   $f:\T\rightarrow\R$ is said to be {\it nabla} differentiable on $\T$ if $f:\T\rightarrow\R$ is {\it nabla} differentiable for all $t\in \T_{\kappa}$.
\end{definition}
  
It is easy to show that, if $f$ is~{\it nabla} differentiable on $\T$, then the following formula holds
$$f^{\rho}=f-\nu f^{\nabla},$$
where $f^{\rho}=f \circ \rho $ (this formula can be seen in~\cite{ABL}).

\begin{definition}
   $f$ is rd-continuously {\it delta} differentiable (we write $f\in C^1_{rd}$) if $f^{\Delta}(t)$ exists for all $t\in\T^k$ and $f^{\Delta}\in C_{rd}$, and
$f$ is ld-continuously {\it nabla} differentiable (we write $f\in C^1_{ld}$) if $f^{\nabla}(t)$ exists for all $t\in\T_k$ and $f^{\nabla}\in C_{ld}$.\end{definition}
  
\begin{remark} If $\T=\R$, then the notion of {\it delta} derivative and {\it nabla} derivative coincide and they denote the standard derivative we know from calculus, however, when $\T=\Z$, then they do not coincide (see~\cite{BohPet}). 
\end{remark}
\section{Dual time scales}
In this section we introduce the definition of {\it dual} time scales. We will see that our main result develops merely from this basic definition. A {\it dual} time scale is just the ``reverse'' time scale of a given time scale. More precisely, we define it as follows:

\begin{definition}
   Given a time scale $\T$ we define the dual time scale $\T^{\star}:= \{ s\in \R | -s\in \T\}$.\end{definition}

Once we have defined a {\it dual} time scale, it is natural to extend all the definitions of Section~\ref{sec2}. We now introduce some notation regarding the correspondence between the definitions on a time scale and  its dual.

Let $\T$ be  a time scale.  If $\rho$ and $\sigma$ denote its associated  jump functions, then we denote by $\hat \rho$ and $\hat \sigma$  the jump functions associated to $\T^{\star}$.
 If $\mu$ and  $\nu$ denote, respectively, the {\it forward graininess} and {\it backward graininess} associated to $\T$, then we denote by $\hat \mu$ and $\hat \nu$, respectively, the  {\it forward graininess} and the {\it backward graininess}  associated to $\T^{\star}$.

Next, we define another fundamental ``dual'' object, i.e., the ``dual'' function.

\begin{definition}
   Given a function $f:\T\rightarrow \R$ defined on time scale $\T$ we define the dual function $f^{\star}:\T^{\star}\rightarrow\R$ on the time scale $\T^{\star}:= \{ s\in \R | -s\in \T\}$ by $f^{\star}(s):=f(-s)$ for all $s\in\T^{\star}$.
\end{definition}
  
\begin{definition}
  \label{lemt} Given a time scale $\T$ we refer to the~{\it delta} calculus (resp.~{\it nabla} calculus) any calculation that involves~{\it delta} derivatives (resp.~{\it nabla} derivatives). 
\end{definition}

\section{Dual correspondences}
In this section we deduce some basic lemmas which follow easily from the definitions. These lemmas concern the relationship between {\it dual} objects. 
We will use the following notation: given the quintuple $(\T, \sigma, \rho, \mu, \nu)$, where
 $\T$ denotes a time scale with jump functions, $\sigma$, $\rho$, and associated  {\it forward graininess} $\mu$ and {\it backward graininess} $\nu$,
 its dual will be
$(\T^{\star}, \hat\sigma, \hat\rho, \hat \mu, \hat\nu)$ where  $\hat\sigma, \hat\rho, \hat \mu,$ and $\hat\nu$ will be given as in Lemma~\ref{lem1} and~\ref{lem2} that we will prove in this section.
Also,  $\Delta$ and $\nabla$ will denote the derivatives for the time scale $\T$ and  $\hat\Delta$ and $\hat\nabla$ will  denote the derivatives for the time scale $\T^{\star}$. 

\begin{lemma} If $a,b\in\T$ with $a<b$,
$$([a,b])^{\star}=[-b,-a].$$

\end{lemma}
 \begin{proof} The proof is straightforward. In fact,
 $$s\in ([a,b])^{\star}\,\,\,\hbox{iff} \,\,\,\,-s\in [a,b] \,\,\,\hbox{iff} \,\,\,\, s\in [-b,-a].$$
    \end{proof}
  
\begin{lemma} \label{lem1} Given $ \sigma$, $\rho:\T\rightarrow \T$, the jump operators for $\T$, then the jump operators for $\T^{\star}$,   $ \hat\sigma$ and $\hat \rho:\T^{\star}\rightarrow\T^{\star}$,
are given by the following two identities:
$$\hat{\sigma}(s)=-\rho(-s),$$
$$\hat{\rho}(s)=-\sigma(-s),$$
for all $s\in\T^{\star}$.
\end{lemma}
 \begin{proof} We show the first identity. Using the definition and some simple algebra,
 $$\hat{\sigma}(s)=\inf\{-w\in\T: -w<-s\}=-\sup\{v\in\T: v<-s\}=-\rho(-s).$$
 The second identity follows similarly.
    \end{proof}

\begin{lemma} Given a time scale $\T$, then 
$$(\T^{\kappa})^{\star}=(\T^{\star})_{\kappa}, \,\,\,\, \hbox{and}\,\,\,\, (\T_{\kappa})^{\star}=(\T^{\star})^{\kappa}.$$
\end{lemma}
  \begin{proof} We first observe that $ \sup\T=-\inf\T^{\star}$.
  
\noindent  If $\sup\T=\infty$, then
   $$(\T^{\kappa})^{\star}=(\T)^{\star}= (\T^{\star})_{\kappa}.$$ 
  If $\sup\T<\infty$, then 
  $$(\T^{\kappa})^{\star}=(\T\setminus (\rho(\sup \T),\sup\T])^{\star}=\T^{\star}\setminus (\rho(\sup \T),\sup\T])^{\star}=(\T^{\star})_{\kappa}.$$
  Similarly, $ (\T_{\kappa})^{\star}=(\T^{\star})^{\kappa}$.
  
    \end{proof}
\begin{lemma} \label{lem2} Given $ \mu:\T\rightarrow \R$, the forward graininess of $\T$, then the backwards graininess of $\T^{\star}$, $\hat \nu:\T^{\star}\rightarrow \R$, is given by the identity
$$\hat\nu (s)=\mu^{\star}(s)\,\,\,\hbox{for\,\,\, all}\,\,\, s\in\T^{\star}.$$
Also, given $ \nu:\T\rightarrow \R$, the backward graininess of $\T$, then the forward graininess of $\T^{\star}$, $\hat \mu:\T^{\star}\rightarrow \R$, is given by the identity
$$\hat\mu (s)=\nu^{\star}(s)\,\,\,\hbox{for\,\,\, all}\,\,\, s\in\T^{\star}.$$
\end{lemma}
 \begin{proof} We prove the first identity, the second will follow analogously. Let $ s\in\T^{\star}$, then
 $$\hat\nu (s)=s-\hat\rho(s)=s+\sigma^{\star}(s)=\mu^{\star}(s).$$
 
    \end{proof}
\begin{lemma} \label{lemmacont}
Given $f:\T\rightarrow \R$, f is rd continuous (resp. ld continuous) if and only if its dual $f^{\star}:\T^{\star}\rightarrow \R$ is ld continuous (resp. rd continuous). 
\end{lemma}

 \begin{proof}  We will only show the statement for rd continuous functions as the proof for ld continuous functions is analogous. We first observe that $t\in \T$ is a  right-dense point iff $-t\in\T^{\star}$ is a left-dense point. Also, $f:\T\rightarrow \R$ is continuous at $t$ iff $f^{\star}:\T^{\star}\rightarrow \R$  is continuous at $-t$.  Let $f:\T\rightarrow \R$ be a function, then, the following is true:

$f:\T\rightarrow \R$\,\,\, is rd continuous iff $f$  is continuous at the right-dense points and its left-sided limits exist (finite) at all left-dense points
 $\hbox {iff}$
 $f^{\star}$ is continuous at the left-dense points and its right-sided limits exist (finite) at all right-dense points
 $\hbox {iff}$
$f^{\star}:\T^{\star}\rightarrow \R$ is ld continuous.
    \end{proof}

The next lemma  links {\it delta} derivatives to {\it nabla} derivatives, showing that the two fundamental concepts of the two types of calculus are, in a certain sense, the dual of each other. In fact, this is the key lemma for our main results. 
\begin{lemma} \label{lemmaderi}Let $f:\T\rightarrow \R$ be {\it delta} (resp.~{\it nabla}) differentiable at $t_0\in\T^{\kappa}$ (resp. at $t_0\in\T_{\kappa}$), then $f^{\star}:\T^{\star}\rightarrow \R$ is {\it nabla} (resp.~{\it delta}) differentiable at $-t_0\in(\T^{\star})_{\kappa}$ (resp. at $-t_0\in(\T^{\star})^{\kappa}$), and the following identities hold true
 $$f^{\Delta}(t_0)=-(f^{\star})^{\hat\nabla}(-t_0)\,\,\,\, (\hbox{resp.}\,\,\,\,  f^{\nabla}(t_0)=-(f^{\star})^{\hat\Delta}(-t_0)),$$
or,
 $$f^{\Delta}(t_0)=-((f^{\star})^{\hat\nabla})^{\star}(t_0)\,\,\,\, (\hbox{resp.}\,\,\,\,  f^{\nabla}(t_0)=-((f^{\star})^{\hat\Delta})^{\star}(t_0)),$$
 or,
  $$(f^{\Delta})^{\star}(-t_0)=-((f^{\star})^{\hat\nabla})(-t_0)\,\,\,\, (\hbox{resp.}\,\,\,\, ( f^{\nabla})^{\star}(-t_0)=-(f^{\star})^{\hat\Delta}(-t_0)),$$
 where $\Delta$, $\nabla$ denote the derivatives for the time scale $\T$ and  $\hat\Delta$, $\hat\nabla$  denote the derivatives for the time scale $\T^{\star}$. 
 \end{lemma}

   \begin{proof}
The proof is trivial but for the sake of completeness we will write all the details.~

We will prove that if $f:\T\rightarrow \R$ is {\it delta} differentiable at $t_0\in\T^{\kappa}$, then $f^{\star}$ is {\it nabla} differentiable at $-t_0\in(\T^{\star})_{\kappa}$.
  Let $f:\T\rightarrow \R$ be {\it delta} differentiable at $t_0\in T^{\kappa}$. Then
   for all $\epsilon>0$ there exists $U$ a neighborhood of $t_0$ such that the inequality
$$|f(\sigma(t_0))-f(s)-f^{\Delta}(t_0)(\sigma(t_0)-s)|  < \epsilon |\sigma(t_0)-s|,$$
is true for all $s\in U$. Next, using Lemma~\ref{lem1}, as well as the definition of dual function $f^{\star}$, we rewrite the above inequality as
$$|f(-\hat\rho(-t_0))-f^{\star}(-s)-f^{\Delta}(t_0)(-\hat\rho(-t_0)-s)|  < \epsilon |-\hat\rho(-t_0)-s|,$$
for all $s\in U$. Let $U^{\star}$ be the dual of $U$. Let $t\in U^{\star}$, then $-t\in U$. Hence, by replacing $s$ by $-t$, we obtain
  $$|(f^{\star}(\hat\rho(-t_0))-f^{\star}(t)-f^{\Delta}(t_0)(-\hat\rho(-t_0)+t)|  < \epsilon |-\hat\rho(-t_0)+t|,$$
    $$|f^{\star}(\hat\rho(-t_0))-f^{\star}(t)-(-f^{\Delta}(t_0))(\hat\rho(-t_0)-t)|  < \epsilon |\hat\rho(-t_0)-t|.$$
   By definition,  this implies that the function $f^{\star}$ is {\it nabla} differentiable at $-t_0$, and 
   $$(f^{\star})^{\hat\nabla}(-t_0)=-f^{\Delta}(t_0).$$
   Analogously, it follows that, if $f:\T\rightarrow \R$ is {\it nabla} differentiable at $t_0\in\T_{\kappa}$, then $f^{\star}:\T^{\star}\rightarrow \R$ is {\it delta} differentiable at $-t_0\in(\T^{\star})^{\kappa}$, and
     $$(f^{\star})^{\hat\Delta}(-t_0)=-f^{\nabla}(t_0).$$
  \end{proof}
  
  The next two lemmas link the notions of $C^1_{rd}$ and $C^1_{ld}$ functions.
  \begin{lemma} \label{lemmacontder}
Given a function $f:\T\rightarrow \R$, $f$ belongs to  $C^1_{rd}$ (resp. $C^1_{ld}$) if and only if its dual $f^{\star}:\T^{\star}\rightarrow \R$ belongs to $C^1_{ld}$ (resp. $C^1_{rd}$) .
\end{lemma}
  \begin{lemma} \label{lemmacontder1}
Given a function $f:\T\rightarrow \R$, $f$ belongs to  $C^1_{prd}$ (resp. $C^1_{pld}$) if and only if its dual $f^{\star}:\T^{\star}\rightarrow \R$ belongs to $C^1_{pld}$ (resp. $C^1_{prd}$) .
\end{lemma}

 In the following example we derive a well known formulas for derivatives. We will deduce the formula for the {\it nabla} derivative using the one for the {\it delta} derivative.
 \begin{example}{(Formula for derivatives.)}
\end{example}
 It is well known~(see~\cite{Boh}) that if $f$ is {\it delta} differentiable on $\T$, with $\mu$ the associated forward graininess, then the formula holds
\begin{equation}\label{derform1}
f^{\sigma}(t)=f(t)+\mu(t) f^{\Delta}(t)\,\,\,\hbox{for\,\,\, all}\,\,\, t\in\T^{\kappa},
\end{equation}
where $f^{\sigma}=f \circ \sigma $.

We will use it to derive the analogous formula for the {\it nabla} derivative. Suppose that $h$ is {\it nabla} differentiable on $\T$ , with $\nu$ its associated  backward graininess, then its dual function $h^{\star}$ is {\it delta} differentiable on $\T^{\star}$. Hence, we apply~(\ref{derform1}) to $h^{\star}$: 

\begin{equation}
( h^{\star})^{\hat \sigma}(s)=h^{\star}(s)+\hat \mu(s) (h^{\star})^{\hat\Delta}(s)\,\,\,\hbox{for\,\,\, all}\,\,\, s\in(\T^{\star})^{\kappa}.
\end{equation}
We observe that $\hat \mu=\nu^{\star}$, while $( h^{\star})^{\hat\sigma}= h^{\rho}$ by Lemma~\ref{lem1}, and Lemma~\ref{lem2}, with $h^{\rho}=h\circ \rho$, and $(h^{\star})^{\hat\Delta}=-h^{\nabla}$ by Lemma~\ref{lemmaderi}. So,

\begin{equation}\label{derform2}
h^{\rho}(t)=h(t)-\nu(t) h^{\nabla}(t)\,\,\hbox{for\,\,\, all}\,\,\, t\in\T_{\kappa}.
\end{equation}
We recall that this formula~(\ref{derform2}) has appeared in the {\it nabla} context in~\cite{ABL}.

Next, using Lemma~\ref{lemmacont} and Lemma~\ref{lemmaderi}, we show in the following proposition how to compare {\it nabla} and {\it delta} integrals.

  \begin{proposition} \label{integral}
(i) Let $f:[a,b]\rightarrow \R$ be a rd continuous, then the following two integrals are equal
  $$\int_a^b f(t)\Delta t=\int_{-b}^{-a}f^{\star}(s)\hat\nabla s;$$
(ii) Let $f:[a,b]\rightarrow \R$ be a ld continuous, then the following two integrals are equal
  $$\int_a^b f(t)\nabla t=\int_{-b}^{-a}f^{\star}(s)\hat\Delta s.$$
   \end{proposition}
  \begin{proof}
Proof of $(i)$. By definition of integral, 
  $$\int_a^b f(t)\Delta t= F(b)-F(a),\,\,\,\,\hbox{where}\,\,\, F$$
   is an antiderivative of $f$,\,i.e.,\,\,\,\,\,
 $$ F^{\Delta}(t)=f(t).$$
 We have seen in Lemma~\ref{lemmaderi} that : $f^{\star}(s)=(F^{\Delta})^{\star}(s)=-(F^{\star})^{\hat\nabla}(s)$. Also, again by definition,
 $$ \int_{-b}^{-a}f^{\star}(s)\hat\nabla s=G(-a)-G(-b),\,\,\,\,\hbox{where $G$}$$
  is an antiderivative of $f^{\star}$, i.e., 
  $$G^{\nabla}(s)=f^{\star}(s).$$
 
It follows that: $G=-F^{\star}+c$, where $c\in\R$, and
 $$ \int_{-b}^{-a}f^{\star}(s)\hat\nabla s=-F^{\star}(-a)+F^{\star}(-b)=-F(a)+F(b)=\int_a^b f(t)\Delta t.$$
Proof $(ii)$.  We apply $(i)$ to $f^{\star}$,
$$\int_{-b}^{-a} f^{\star}(s)\hat\nabla s=\int_{a}^{b}(f^{\star})^{\star}(t)\nabla t.$$
Since $(f^{\star})^{\star}=f$, $(ii)$ follows immediately.
 \end{proof}

\section{Main Result}
The main result of this article will be the following Duality Principle which asserts that given certain results in the {\it nabla} (resp. {\it delta}) calculus under certain hypotheses, one can obtain the dual results by considering the corresponding dual hypotheses and the dual conclusions in the {\it delta} (resp. {\it nabla}) setting.

Given a statement in the~{\it delta} calculus  (resp.~{\it nabla} calculus), the corresponding dual statement is obtained by replacing any object in the given statement by the corresponding dual one.

\noindent{\bf Duality Principle} 
{\it For any statement true in the {\it nabla} (resp. {\it delta}) calculus in the time scale $\T$ there is an equivalent dual statement in the {\it delta} (resp.{\it nabla}) calculus for the dual time scale $\T^{\star}$. }

In the next example we further illustrate how the Duality Principle applies.

\begin{example}{(Integration by parts.)}
\end{example}
 We show how the Duality Principle can be applied to prove the integration by parts formula.
In {\it delta} settings the integration by parts formula is given by the following identity:
\begin{equation}\label{intparts}
\int_a^b f(t)g^{\Delta}(t)\Delta t=f(b)g(b)-f(a)g(a)-\int_a^b f^{\Delta}(t)g^{\bar \sigma}(t)\Delta t,
\end{equation}
for all functions $f,g:[a,b]\rightarrow \R$, with $f,g\in C^1_{rd}$. 

Now, let $h,j:[a,b]\rightarrow \R$, with $h,j\in C^1_{ld}$, then, the dual functions $h^{\star}$, $j^{\star}:[-b,-a]\rightarrow \R$ are in $C^1_{rd}$. 

Next, we will apply the identity~(\ref{intparts}) to  $h^{\star}$ and $j^{\star}$:

$$ \int_{-b}^{-a} h^{\star}(t) (j^{\star})^{\hat\Delta}(t)\hat\Delta t=
 h^{\star}(-a)j^{\star}(-a)-h^{\star}(-b)j^{\star}(-b)- \int_{-b}^{-a} (h^{\star})^{\hat\Delta}(t) (j^{\star})^{\sigma}(t)\hat\Delta t.
$$
The LHS of the last identity can be written as:
\begin{equation}\label{intpartsa}
\int_{-b}^{-a} h^{\star}(t) (j^{\star})^{\hat\Delta}(t)\hat\Delta t=-\int_{-b}^{-a}(h\,j^{\nabla})^{\star}(t)\hat\Delta t=-\int_a^b h(s)\, j^{\nabla} (s)\nabla s,
 \end{equation}
because $(h\,j^{\nabla})^{\star}(t)=h^{\star}(t) (j^{\nabla})^{\star}(t)= -h^{\star}(t) (j^{\star})^{\hat\Delta}(t)$.

The second term in the RHS can be written as:

\begin{equation}\label{intpartsb}
 \int_{-b}^{-a} (h^{\star})^{\hat\Delta}(t) (j^{\star})^{\hat \sigma}(t)\Delta t=\int_a^b ( (h^{\star})^{\hat\Delta}(j^{\star})^{\hat \sigma})^{\star}(s)\nabla s= -\int_a^b h^{\nabla}(s) j^{\rho}(s)\nabla s,
 \end{equation}
because of the identity $((j^{\star})^{\hat\sigma})^{\star}(s)=j^{\rho}(s)$. 

To obtain the desired formula we substitute the RHS of~(\ref{intpartsb}) in the integration by parts formula~(\ref{intparts}):
\begin{equation}\label{intpartsna}
\int_a^b h(s)\, j^{\nabla} (s)\nabla(s) =-h(a)j(a)+h(b)j(b)-\int_a^b h^{\nabla}(s) j^{\rho}(s)\nabla s.
\end{equation}
It follows that the  identity~(\ref{intpartsna}) is the integration by parts formula for the {\it nabla} setting.

\section{Application of the Duality Principle to the calculus of variations on time scales}

\subsection{ Euler-Lagrange equation}

We consider the Euler-Lagrange equation using the identity of Proposition~\ref{integral}.~We will use Bohner's results in~\cite{Boh} in the {\it delta} settings to prove similar results in the {\it nabla} settings as done in~\cite{ABL} (one could also do the vice versa). We review a few definitions.  

\begin{definition}
   A function $f:[a,b]\rightarrow \R$ belongs to the space $C^1_{rd}$ if the following norm is finite: $||f||_{C^1_{rd}}=||f||_{0,r}+\max_{t\in [a,b]^{\kappa}} |f^{\Delta}(t)|$, where $||f||_{0,r}=\max_{t\in [a,b]^{\kappa}} |f^{\sigma}(t)|$; also,
 a function $f:[a,b]\rightarrow \R$ belongs to the space $C^1_{ld}$ if the following norm is finite: $||f||_{C^1_{ld}}= ||f||_{0,l}+\max_{t\in [a,b]_{\kappa}} |f^{\nabla}(t)|$, where $||f||_{0,l}=\max_{t\in [a,b]_{\kappa}} |f^{\rho}(t)|$.
\end{definition}

\begin{definition}
   A function $f$ is {\it delta} regulated if the right-hand limit $f (t +)$ exists (finite) at all right-dense points $t\in \T$  and the left-hand limit $f (t -)$ exists at all left-dense points  $t\in \T$; 
$f$ is regulated if the left-hand limit $f (t +)$ exists (finite) at all left-dense points $t\in \T$  and the right-hand limit $f (t -)$ exists at all right-dense points  $t\in \T$.\end{definition}

\begin{definition}
  A function $f$ is {\it delta} piecewise rd-continuous (we write $f \in C_{prd}$ ) if it is regulated and if it is rd continuous at all,  except possibly at Þnitely many, right-dense points  $t\in \T$; $f$ is {\it nabla} piecewise ld-continuous (we write $f \in C_{pld}$ ) if it is {\it nabla} regulated and if it is ld continuous at all,  except possibly at Þnitely many, left-dense points  $t\in \T$.\end{definition}
  
\begin{definition}
   $f$ is {\it delta} piecewise rd-continuously differentiable (we write $f \in C^1_{prd}$ ) if $f$ is rd continuous and $f^{\Delta}\in C_{prd}$;
 $f$ is {\it delta} piecewise ld-continuously differentiable (we write $f \in C^1_{pld}$ ) if $f$ is ld continuous and $f^{\nabla}\in C_{pld}$.\end{definition}

\begin{definition}
   Assume the function $L:\T\times \R\times\R\rightarrow \R$  is of class $C^2$ in the second and third variable, and $rd$ continuous in the first variable. Then, $y_0$ is said to be a weak (resp. strong) local minimum of the problem
\begin{equation}\label{variprob1}
\mathcal L(y)= \int_a^b L(t, y^{\sigma}(t), y^{\Delta}(t))\Delta t\,\,\,\,\,\, y(a)=\alpha, \,\,\, y(b)=\beta,
\end{equation}
where $a$, $b\in\T$, with $a<b$; $\alpha$, $ \beta\in\R$, and $L:\T\times \R\times\R\rightarrow \R$,

if $y_0(a)=\alpha, \,\, y_0(b)=\beta,$ and $\mathcal L(y_0)\leq \mathcal L(y)$ for all $y\in C^{1}_{rd}$ with  $y(a)=\alpha, \,\,\, y(b)=\beta$  and $||y-y_0||_{C^1_{rd}}\leq \delta$ (resp.  $||y-y_0||_{0,r}\leq \delta$) for some $\delta>0$.
 \end{definition}
  
 We refer to the function $L$ as to the Lagrangian for the above problem.
 Moreover, if $L=L(t,x,v)$, then $L_v$, $L_x$ represent, respectively, the partial derivatives of $L$ with respect to $v$, and $x$.

\begin{definition}
   Assume the function $\bar L:\T\times \R\times\R\rightarrow \R$  is of class $C^2$ in the second and third variable, and $rd$ continuous in the first variable. Then, $y_0$ is said to be a weak (strong) local minimum of the problem
\begin{equation}\label{variprob2}
\bar {\mathcal L}(h)= \int_{c}^{d} \bar L(s, h^{\rho}(s), h^{\nabla}(s))\nabla s\,\,\,\,\,\, h(c)=A, \,\,\, h(d)=B,
\end{equation}
where $c$, $d\in\T$, with $c<d$; $A$, $ B\in\R$, and $\bar L:\T\times \R\times\R\rightarrow \R$,

if $y_0(c)=A, \,\, y_0(d)=B,$ and $\bar{\mathcal L}(y_0)\leq\bar{ \mathcal L}(y)$ for all $y\in C^{1}_{ld}$ with  $y(c)=A, \,\,\, y(d)=B$  and $||y-y_0||_{C^1_{ld}} \leq \delta$ (resp.  $||y-y_0||_{0,l} \leq \delta$)  for some $\delta>0$.
 \end{definition}
  
 \begin{definition}
   Given a Lagrangian  $L:\T\times \R\times\R\rightarrow \R$, we define the dual (corresponding) Lagrangian $L^{\star}:\T^{\star}\times \R\times\R\rightarrow \R$ by 
   $L^{\star}(s,x,v)=L(-s,x,-v)$ for all $(s,x,v)\in\T^{\star}\times\R\times\R$. 
\end{definition}
  
As a consequence of the definition of the dual Lagrangian and Proposition~\ref{integral} we have the following useful lemma:
\begin{lemma}
\label{corlag} Given a Lagrangian $L:[a,b]\times \R\times\R\rightarrow \R$,  then the following identity holds:
$$\int_a^b L(t, y^{\sigma}(t), y^{\Delta}(t))\Delta t= \int_{-b}^{-a} L^{\star}(s, (y^{\star})^{\hat \rho}(s), (y^{\star})^{\hat\nabla}(s))\hat\nabla s,$$
for all functions $y\in C^1_{rd} ([a,b])$.

\end{lemma}
Next theorem is a result by Bohner~\cite{Boh} in one dimension (the results we will present can be obtained without this restriction, but we prefer one dimension to have an immediate comparison with the results in~\cite{ABL}).
\begin{theorem} (Euler-Lagrange Necessary Condition in Delta Setting).  
\label{eulerboh}
If $y_0$ is 
a (weak) local minimum of the variational problem~(\ref{variprob1}),
then the Euler-Lagrange equation
$$L_v^{\Delta}(t, y_0^{\sigma}(t), y_0^{\Delta}(t))=L_x(t, y_0^{\sigma}(t), y_0^{\Delta}(t)), \,\,\,\,\hbox{for all}\,\,\,\, t\in[a,b]^{\kappa},$$
holds.

\end{theorem}
Now, we will use Bohner's theorem to prove the Euler-Lagrange equation in the {\it nabla} context. We recall that the Euler-Lagrange equation in the {\it nabla} context  was shown in~\cite{ABL}. Here we will reprove it using our technique. (Also, see Remark~\ref{inter}.)

\begin{theorem}  (Euler-Lagrange Necessary Condition in Nabla Setting).
\label{ELratici} 
If $\bar y_0$ is a local (weak) minimum for the variational problem~(\ref{variprob2}), 
then the Euler-Lagrange equation
 $${\bar {L}}_x(s, (\bar y_0)^{\rho}(s), (\bar y_0)^{\nabla}(s))=(\bar {L}_w)^{\nabla}(s, (\bar y_0)^{ \rho}(s), (\bar y_0)^{\nabla}(s))\,\,\,\,\hbox{for all}\,\,\,\, s\in[c,d]_{\kappa},$$
holds.

\end{theorem}
\begin{proof} 
This theorem is essentially a corollary of Theorem~\ref{eulerboh}.
Since $\bar y_0$ is a local minimum for ~(\ref{variprob2}), it follows from Lemma~\ref{corlag}  that $\bar y_0^{\star}$ is local minimum for the variational problem
\begin{equation}\label{variprob3}
(\bar {\mathcal L})^{\star}(g)= \int_{-d}^{-c} \bar L^{\star}( t, g^{\hat\sigma}(t), g^{\hat\Delta}(t))\hat\Delta t,\,\,\,\,\,\, g(-c)= A,\,\,\, g(-d)=B,
\end{equation}
where $g\in C^1_{rd}$.

The variational problem~(\ref{variprob3}) is the same as~(\ref{variprob1}) for the Lagrangian $\bar L^{\star}$ (with $a=-d$, $b=-c$, $\alpha=B$ and $\beta=A$). Hence, we can apply Theorem~\ref{eulerboh}. The Euler-Lagrange equation  for the Lagrangian $\bar L^{\star}$ is given by :
\begin{equation}\label{ELBoh}(\bar L^{\star}_v)^{\hat\Delta}(t, (\bar y_0^{\star})^{\hat \sigma}(t), (\bar y_0^{\star})^{\hat\Delta}(t))=\bar L^{\star}_x(t, (\bar y_0^{\star})^{\hat \sigma}(t), (\bar y_0^{\star})^{\hat\Delta}(t)), \,\,\,\,\hbox{for all}\,\,\,\, t\in[-d,-c]^{\kappa}.
\end{equation}
Our goal is now to rewrite~(\ref{ELBoh}) for the Lagrangian $\bar L$. It is easy to check that:
$$\bar L^{\star}_v(t,x,v)=-\bar L_w(-t,x, -v),\,\,\,\,\hbox{and}\,\,\,\, \bar L^{\star}_x(t,x,v)=\bar L_x(-t,x, -v),$$
where $\bar L_w$ is the partial derivative of $\bar L$ with respect to the third variable. Let us substitute $x$ by $ (\bar y_0^{\star})^{\hat \sigma}(t)$, and $v$ by $(\bar y_0^{\star})^{\hat\Delta}(t)$, in the previous identities. We get:

$$\bar L^{\star}_v(t,(\bar y_0^{\star})^{\hat \sigma}(t),(\bar y_0^{\star})^{\hat\Delta}(t))=-\bar L_w(-t,(\bar y_0)^{\rho}(-t), (\bar y_0)^{\nabla}(-t)),$$
and
$$\bar L^{\star}_x(t,(\bar y_0^{\star})^{\hat \sigma}(t), (\bar y_0^{\star})^{\hat\Delta}(t))=\bar L_x(-t,(\bar y_0)^{\rho}(-t), (\bar y_0)^{\nabla}(-t)).$$
From Lemma~\ref{lemmaderi}, it follows that:
$$g^{\hat\Delta}(t)=p^{\nabla}(-t)\,\,\,\,\hbox{ for all}\,\,\,\, t\in [-d,-c]^{\kappa},$$
where
$$g(t)=\bar L^{\star}_v(t,(\bar y_0^{\star})^{\hat \sigma}(t), (\bar y_0^{\star})^{\hat\Delta}(t)) \,\,\,\hbox{and}\,\,\,\ p(-t)=\bar L_w(-t,(\bar y_0)^{\rho}(-t), (\bar y_0)^{\nabla}(-t)).$$
Next, let $s\in[c,d]_{\kappa}$ and set $-t=s$. Then by~(\ref{ELBoh}),
\begin{equation}\label{eqnew}
p^{\nabla}(s)=\bar L_x(s,(\bar y_0)^{\rho}(s), (\bar y_0)^{\nabla}(s)),
\end{equation}
and, finally, revealing the definition of $p$, from~(\ref{eqnew}) we obtain the Euler-Lagrange equation in the {\it nabla} setting:
$${\bar {L}}_x(s, (\bar y_0)^{\rho}(s), (\bar y_0)^{\nabla}(s))=(\bar {L}_w)^{\nabla}(s, (\bar y_0)^{ \rho}(s), (\bar y_0)^{\nabla}(s))\,\,\,\,\hbox{for all}\,\,\,\, s\in[c,d]_{\kappa}.$$
\end{proof}

\begin{remark}\label{inter} Theorem~\ref{ELratici} states the same result as the main theorem proven in~\cite{ABL}.  The only difference is the interval of points for which the Euler-Lagrange equation holds.
 In fact, since in~\cite{ABL} the interval of integration for the Lagrangian is  $[\rho^2(a)), \rho(b)]$, it follows from our results that  the Euler-Lagrange equation has to hold in the interval $[\rho^2(a)), \rho(b)]_{\kappa}$ and not $[\rho(a)), b]$ as in~\cite{ABL}. This claim can be also justified by noticing that,  in order of applying~Lemma~$2.1$ in~\cite{ABL}, the test functions have to vanish at the limit points of integration. Another observation about such  interval was pointed out in~\cite{Tor}.
 \end{remark}

\begin{remark} Theorem~\ref{ELratici} can be easily generalized to the higher-order 
results of~\cite{tormar} by applying our Duality Principle to the results in \cite{Torfer}.
 \end{remark}
\subsection{Weierstrass Necessary Condition on Time Scales}
 
We first review a few definitions. Let $L$ be a Lagrangian.  
Let $E : [a , b ]^{\kappa}\times\R^3 \rightarrow \R$ be the function defined as 
$$E (t , x, r, q) = L (t , x, q) - L (t , x, r) - (q - r)  L_r (t , x, r). $$
This function $E$ is called the Weierstrass excess function of $L$.

The Weierstrass necessary optimality condition on time scales was proven in the {\it delta} setting in~\cite{tormal}. Their theorem states as follows:
\begin{theorem} \label{weiteo}(Weierstrass Necessary Optimality Condition with Delta Setting).
\label{weidel}
Let $\T$ be a time scale, $a$ and $b\in \T$, $a < b$ . Assume that the function $L(t , x, r)$ in~(\ref{variprob1}) satisfies the following condition: 
\begin{equation}\label{weiineqdel}
\mu (t ) L (t , x,\gamma r_1 + (1 -\gamma)r_2 ) \leq  \mu(t )\gamma L(t , x, r_1 ) + \mu(t )(1 - \gamma) L (t , x, r_2 ), 
\end{equation}
$\,\,\,\hbox{for each}\,\,\,\, (t , x)\in [a , b ]^{\kappa} \times \R, \,\,\, \hbox{all} \,\,r_1 , r_2\in\R, \gamma \in [0, 1].$

Let  $\bar x$ be a piecewise continuous function. If $\bar{x}$ is a 
strong local minimum for~(\ref{variprob1}), then 
$$E [t , \bar x^{\sigma}(t ), \bar{x}^ {\Delta}(t ), q]\geq 0
\,\,\,\hbox{ for all}\,\,\,\, t \in [a , b ]^{\kappa} \,\,\,\hbox{and}\,\,\, q\in\R,$$
where we replace $\bar x^{\Delta}(t )$ by $\bar x^{\Delta}(t -)$ and  
$\bar x^{\Delta} (t +)$  at finitely many points $t$ where  
$\bar x^{\Delta} (t )$ does not exist. 
\end{theorem}

Let $E$ be  the Weierstrass excess function of $\bar L$. 
\begin{theorem} (Weierstrass Necessary Optimality Condition with Nabla Setting).
Let $\T$ be a time scale, $a$  and $b\in \T$, $a < b$ . Assume that the function $\bar L (t , x, r)$ in~(\ref{variprob2}) satisfies the following condition: 
\begin{equation}\label{weiineqnab}
\nu (t ) \bar L (t , x,\gamma r_1 + (1 -\gamma)r_2 ) \leq  \nu(t )\gamma \bar L (t , x, r_1 ) + \nu(t )(1 - \gamma)\bar  L (t , x, r_2 ),
\end{equation}
$\,\,\,\hbox{for each}\,\,\,\, (t , x)\in [a , b ]_{\kappa} \times \R, \,\,\, \hbox{all} \,\,r_1 , r_2\in\R, \gamma \in [0, 1].$

Let  $\bar x$ be a piecewise continuous function. If $\bar{x}$ is a 
strong local minimum for ~(\ref{variprob2}), then 
$$E [t , \bar x^{\rho}(t ), \bar{x} ^{\nabla}(t ), q]\geq 0
\,\,\,\hbox{ for all}\,\,\,\, t \in [a , b ]_{\kappa} \,\,\,\hbox{and}\,\,\, q\in\R$$
where we replace $\bar x^{\nabla}(t )$ by $\bar x^{\nabla}(t -)$ and  
$\bar x^{\nabla}(t +)$  at finitely many points $t$ where  
$\bar x^{\nabla} (t )$ does not exist. 
\end{theorem}
\begin{proof} 
Let $\bar L^{\star}$ be the dual Lagrangian of $\bar L$. It is easy to prove (similarly as we did in Theorem~\ref{ELratici}, although here $\bar x$ is a strong minimum),  that $\bar{x}^{\star}$ is a strong local minimum for ~(\ref{variprob1}). 
Then,~(\ref{weiineqnab}) can be written on the dual  time scale $\T^{\star}$ as
$$\bar\mu (s) \bar L^{\star} (s , x,-\gamma r_1 -(1 -\gamma)r_2 ) \leq  \bar\mu(s)\gamma \bar L^{\star} (s , x,- r_1 ) + \bar\mu(s)(1 - \gamma) \bar L^{\star} (s , x, -r_2 ),$$
$\,\,\,\hbox{for each}\,\,\,\, (s , x)\in [-b , -a ]^{\kappa} \times \R, \,\,\, \hbox{all} \,\,r_1 , r_2\in\R, \gamma \in [0, 1].$

We recognize that the last inequality is  the same as~(\ref{weiineqdel}) in Theorem~\ref{weidel} for the Lagrangian $L^{\star}$. Hence, we apply Theorem~\ref{weidel},
$$E^{\star} [s , (\bar x^{\star})^{\hat \sigma}(s), (\bar{x}^{\star})^ {\hat\Delta}(s ), q]\geq 0
\,\,\,\hbox{ for all}\,\,\,\,s\in [-b , -a ]^{\kappa} \,\,\,\hbox{and}\,\,\, q\in\R,$$
where $E^{\star} $ is the Weierstrass excess function of $\bar L^{\star}$. 

Also, we notice that
$$E^{\star} [s , (\bar x^{\star})^{\hat \sigma}(s), (\bar{x}^{\star})^ {\hat\Delta}(s ), q]=E[-s,  (\bar x^{\star})^{\hat \sigma}(s), -(\bar{x}^{\star})^ {\hat\Delta}(s ), -q],$$
 where $E$ is the Weierstrass excess function of $\bar L$. 

Finally,
$$E [t , \bar x^{\rho}(t ), \bar{x}^ {\nabla}(t ), -q]\geq 0\,\,\,\hbox{ for all}\,\,\,\, t \in [a , b ]_{\kappa} \,\,\,\hbox{and all}\,\,\, q\in\R,$$
because
$$(\bar x^{\star})^{\hat \sigma}(s)= \bar x^{\rho}(-s ), \,\,\, \hbox{and}\,\,\,  (\bar{x}^{\star})^ {\hat\Delta}(s )=-\bar{x}^ {\nabla}(-s).$$
We observe that, the fact that  we can replace 
$$\bar x^{\nabla}(t )$$
 by
$$\,\,\,\, \bar x^{\nabla}(t -)\,\,\, \hbox{ and}\,\,\,\,  \bar x^{\nabla}(t +)\,\,\,\,\hbox{ at finitely many points}\,\,\, t,$$ 
where  
$$\bar x^{\nabla} (t )$$
does not exist, follows as well from Theorem~\ref{weiteo}.
\end{proof}
\vskip 1cm

\newpage
\begin{appendix}
\section{Table of dual objects}
Based on the above definitions, remarks and lemmas we summarize in the following table for each ``object'' its dual one. Naturally, this table may be extended to more objects. \\
 \renewcommand{\arraystretch}{0.8}  
  \begin{table}[h!b!p!]
  \caption{Dual objects}

  \begin{tabular}{|c|c|clcl}    \hline
    \textbf{Object} & \textbf{Corresponding dual object}\\ \hline
    $\T$             &       $\T^{\star}$                  \\ \hline
     $f:\T\rightarrow \R$            &    $f^{\star}:\T^{\star}\rightarrow\R$                              \\ \hline
      $f^{\star}:\T^{\star}\rightarrow\R$               &   $f:\T\rightarrow \R$                               \\ \hline
    $t_0$ right-dense (left-dense)             &      $-t_0$ left-dense (right-dense )                      \\ \hline
  
   $t_0$ right-scattered (left-scattered) &  $-t_0$ left-scattered (right-scattered)  \\ \hline
  
 $ \mu$,  $ \nu$  & $\hat \nu(=\mu^{\star}) $, $\hat \mu(=\nu^{\star})$ \\ \hline

 $\sigma$,  $\rho$ & $\hat \rho(=-\sigma^{\star}$),  $\hat\sigma(=-\rho^{\star}$) \ \\ \hline

 $f^{\Delta}(t_0)$ & $-(f^{\star})^{\hat\nabla}(-t_0)$ \\ \hline
 $f^{\nabla}(t_0)$ & $-(f^{\star})^{\bar \Delta}(-t_0)$ \\ \hline
 $f^{\Delta}(t_0)$ & $-((f^{\star})^{\bar \nabla})^{\star}(t_0)$ \\ \hline
 $(f^{\Delta})^{\star}(-t_0)$ & $-((f^{\star})^{\hat\nabla})(-t_0))$ \\ \hline
 $f\in C_{rd}$ ( $f\in C_{ld}$) &  $f^{\star}\in C_{ld}$ ($f^{\star}\in C_{rd}$) \\ \hline

  $f\in C^1_{rd}$ ( $f\in C^1_{ld}$) &  $f^{\star}\in C^1_{ld}$ ($f^{\star}\in C^1_{rd}$) \\ \hline

 $f\in C_{prd}$ ( $f\in C_{pld}$) &  $f^{\star}\in C_{pld}$  ( $f^{\star}\in C_{prd}$)\\ \hline

 $f\in C^1_{prd}$ ($f\in C^1_{pld}$)&  $f^{\star}\in C^1_{pld}$( $f^{\star}\in C^1_{prd}$)  \\ \hline

 $\int_a^b f(t)\Delta t$ & $\int_{-b}^{-a}f^{\star}(s)\hat\nabla s$ \\ \hline
$L:\T\times \R^2\rightarrow \R$, $L(t,x,v)$ & $L^{\star}:\T^{\star}\times \R^2\rightarrow \R$, $L^{\star}(s,x,w) (=L(-s,x,-w))$\\ \hline 
  \end{tabular}
\end{table}

{\bf Acknowledgments.}~The author would like to thank Professors D.~Torres and A.~Malinowska for having brought this problem to her attention while they were visiting the University of Texas, at Austin, in the Fall 2009. Also, she thanks them for patiently reading some drafts of this article and making very helpful suggestions.

\end{appendix}

\end{document}